\documentclass[12pt]{amsart}
\usepackage{amsmath,amssymb,amsfonts,amsthm,amsopn}
\usepackage{graphics}

 \def\Spnr{Sp(d,\R)}
 
 \newcommand\mc[1]{{\mathcal{#1}}}

\newcommand{\tfa}{time-frequency analysis}

\newcommand{\stft}{short-time Fourier transform}

\newcommand{\fif}{if and only if}
\newcommand{\tfs}{time-frequency shift}

\newcommand{\modsp}{modulation space}

\newtheorem{theorem}{Theorem}[section]
\newtheorem{lemma}[theorem]{Lemma}

\newtheorem{proposition}[theorem]{Proposition}
\newtheorem{definition}[theorem]{Definition}

\newcommand{\beqa}{\begin{eqnarray*}}
\newcommand{\eeqa}{\end{eqnarray*}}

\newcommand{\field}[1]{\mathbb{#1}}
\newcommand{\bR}{\field{R}}        
\newcommand{\bZ}{\field{Z}}        
\newcommand{\bC}{\field{C}}        
\def\G{\mathcal{G}}

\def\la{\lambda}

\def\eps{\epsilon}

\def\spdr{{\mathfrak {sp}}(d,\bR)}
\def\cS{\mathcal{S}}

\def\cM{\mathcal{M}}

\def\cA{\mathcal{A}}

\def\cC{\mathcal{C}}

\def\wpr{WF^{p,r}_G}

\def\a{\aleph}

\def\rd{\bR^d}

\def\rdd{{\bR^{2d}}}

\def\lrd{L^2(\rd)}

\def\intrd{\int_{\rd}}
\def\intrdd{\int_{\rdd}}

\def\R{\right)}

\def\<{\left<}
\def\>{\right>}

\def\inv{^{-1}}

\def\mv1{M_v^1}

\def\phas{(x,\xi)}
\def\mn{(m,n)}
\def\mn'{(m',n')}

\def\Spnr{Sp(d,\R)}

\def\o{\eta}
\def\a{\alpha}

\def\R{\mathbb{R}}
\def\Ren{\mathbb{R}^d}
\def\Renn{\mathbb{R}^{2d}}

\def\sch{\mathcal{S}}

\def\Fur{\mathcal{F}}

\def\H{{\mathbb H}}

\def\f{\varphi}

\def\Sn2{S_{2}(L^{2}(\Ren))}
\def\S1{S_{1}(L^{2}(\Ren))}
\def\sig00{\sigma_{0,0}}

\def\la{\langle}
\def\ra{\rangle}

\newcommand{\A}{\mathcal{A}}

\begin{document}
\begin{abstract} This work deals with Schr\"odinger equations with
quadratic and sub-quadratic Hamiltonians perturbed by a potential. In particular we shall focus on  bounded, but not necessarily smooth
perturbations, following the footsteps of the preceding works \cite{MetapWiener13,wavefrontsetshubin13}. To the best of our knowledge these are
the pioneering papers which contain the most general results about the time-frequency concentration of the Schr\"odinger evolution. We shall give a representation of such evolution as the composition of a metaplectic operator and a pseudodifferential operator having symbol in certain classes of modulation spaces. About propagation of singularities, we use a new notion of wave front set, which allows the expression of optimal results of propagation in our context. To support this claim, many comparisons with the existing literature are performed in this work.\end{abstract}

\title{On the Schr\"odinger equation with potential in modulation spaces}

\author{Elena Cordero}
\address{Universit\`a di Torino, Dipartimento di Matematica, via Carlo Alberto 10, 10123 Torino, Italy}
\email{elena.cordero@unito.it}
\author{Fabio Nicola}
\address{Dipartimento di Scienze Matematiche,
Politecnico di Torino, corso Duca degli Abruzzi 24, 10129 Torino,
Italy}
\email{fabio.nicola@polito.it}

\maketitle

\section{Introduction}
A wave packet is a function on $\rd$ that is well-localized  both in its time and  frequency domain. Willing to decompose a function $f(x)$ into uniform wave packets, we are led to the following time-frequency concepts, at the basis of the Time-Frequency Analysis:   the linear
operators of translation and modulation
\begin{equation*}
 T_yf(x)=f(x-y)\quad{\rm and}\quad M_{\xi}f(x)= e^{2\pi i \xi
  x}f(x), \quad x,y,\xi\in\rd,
\end{equation*}
whose composition is  the time-frequency shift
 $\pi(z)=M_{\xi} T_{y}$, $z=(y,\xi)\in T^*\rd=\rdd$, called the phase-space (or time-frequency space).\par
Given  $g\in \cS(\rd)\setminus \{0\}$, the so-called Gabor atom $\pi(z)g$ can then  be considered a wave packet. \par
As elementary application of this decomposition, consider  the Cauchy problem for the
Schr\"odinger equation which describes the free particle:
\begin{equation}\label{cp}
\begin{cases}
i\displaystyle\frac{\partial u}{\partial t}+\Delta u=0\\
u(0,x)=u_0(x),
\end{cases}
\end{equation}
with $x\in\R^d$, $d\geq1$. The  explicit formula for the solution is the Fourier multiplier
\begin{equation}\label{sol}
u(t,x)=e^{i t \Delta}u_0(x)=(K_t\ast u_0)(x),
\end{equation}
where
\begin{equation}\label{chirp0}
K_t(x)=\frac{1}{(4\pi i t)^{d/2}}e^{i|x|^2/(4t)}.
\end{equation}
Starting with the wave packet $u_0= \pi(z)\f$, where $ \f(x)=e^{-\pi |x|^2}$ is the Gaussian function, the computations developed in \cite[Section 6]{fio3} shows
that the solution $u(t,x)=e^{i t \Delta}(\pi(z)\f)(x)$, $z=(z_1,z_2)\in\rdd$, is given by
\begin{equation}\label{zz1}
u(t,x)=(1+4\pi i t)^{-d/2}e^{-\frac{4\pi^2 t z_2 (2z_1+i z_2 )}{1+ 4 \pi i t}}
M_\frac{z_2}{1+ 4 \pi i t}T_{z_1}e^{-\frac{\pi}{1+ 4 \pi i t}|x|^2}.
\end{equation}
Then straightforward computations show there exist constant $C_t>0$, $\eps_t>0$, depending on $t$, such that
\begin{equation}\label{4bis}
|\la u, \pi(w)\f\ra|\leq C e^{-\eps |w-\A_t z|^2}
\end{equation}
where $\A_t z=(z_1+ 4\pi t z_2, z_2)$.
So the evolution $e^{it \Delta}$ sends wave packets into wave packets and the phase-space concentration of the evolution is well-determined.\par
This simple example gives the intuition that the evolution of Schr\"odinger   equations with Hamiltonians coming from classical mechanics can be still interpreted as trajectory of particles in phase-space. The particles corresponds to wave packets that have a good time-frequency localization.

The natural concept in this order of ideas is the time-frequency representation called the short-time Fourier transform.
Precisely, the short-time Fourier transform (STFT) of a function or
distribution $f$ on $\rd$ with respect to a Schwartz window
function $g\in\cS(\rd)\setminus\{0\}$ is defined by
\begin{equation}\label{STFT}
V_g f(z)=\la f ,\pi(z)g\ra =\intrd f(v) \overline{g(v-z_1)}
e^{- 2\pi i v z_2}\, dv,\quad z=(z_1,z_2) \in\rdd.
\end{equation}
The time-frequency localization of $f$ is obtained by comparison to wave packets $\pi(z)g$:  $f$ is localized near $z_0\in\rdd$ if $V_g f(z)$ is large for $z$ near $z_0$ and decays off $z_0$. The formal theory of time-frequency localization leads to modulation spaces, introduced by Feichtinger  in 1983 \cite{F1}.
Namely, fix $g\in\cS(\rd)$, consider a weight function $m$ on $\rdd$ and $1\leq p\leq \infty$. The modulation space $M^p_m(\rd)$
 is  defined as the space of all tempered distribution $f\in\cS'(\rd)$ for which
 \begin{equation}\label{minfty}
  \intrdd |V_g f(z)|^p m(z)^{p}\,dz<\infty
 \end{equation}
 (with obvious modifications for $p=\infty$, see Section $2$ below).
The measure of the time-frequency localization of $f$ is done by computing a weighted $L^p$ norm of the corresponding STFT $V_g f$. If $m=1$ (unweighted case), we simply write $M^p$ in place of $M^p_m$. A preliminary question is then the following.\par\smallskip
\emph{Does the Schr\"odinger evolution preserve modulation spaces?}\par\smallskip
A first positive answer for the evolution of the free particle in \eqref{cp} is contained in the pioneering work \cite{Benyi} on boundedness on modulation spaces for Fourier multipliers:
$$\|e^{i t \Delta}u_0\|_{M^p}\leq C (1+|t|)^{d/2} \|u_0\|_{M^p},\quad u_0\in\cS(\rd).
$$
The proof follows from \eqref{zz1}, \eqref{4bis}, by taking into account the $t$-dependence.

The study of nonlinear Schr\"{o}dinger equations and other PDEs in modulation spaces has been widely developed by B. Wang and collaborators in many papers, see e.g., \cite{RSW,baoxiang,wh} and the recent textbook \cite{Wangbook}.

Here our attention is addressed to more general linear Schr\"odinger   equations. Namely we  show how time-frequency analysis can be successfully applied in the study of the Cauchy problem  for linear Schr\"odinger   equations of the type
  \begin{equation}\label{C1}
\begin{cases} i \displaystyle\frac{\partial
u}{\partial t} +H u=0\\
u(0,x)=u_0(x),
\end{cases}
\end{equation}
with $t\in\bR$ and the initial condition $u_0\in\cS(\rd)$ or some
larger space.  We consider  an operator $H$ of the form
\begin{equation}\label{C1bis}
H=a^w+ \sigma^w,
\end{equation}
 where $a^w$ is the Weyl quantization of a real quadratic homogeneous   polynomial  on
$\rdd$ and  $\sigma^w$ is a pseudodifferential operator (in the Weyl form) with a rough symbol $\sigma$, belonging  to a suitable modulation space.
We recall that  the Weyl quantization of a symbol $a(x,\xi)$ is correspondingly defined as
\[
a^w f(x)=a^w(x,D) f=\iint_{\rdd} e^{2\pi i(x-y)\xi} a\Big(\frac{x+y}{2},\xi\Big) f(y) dy\, d\xi.
\]
The aim is to find conditions on the perturbation $\sigma^w$ for which the evolution $e^{it H}$ preserves modulation spaces.
Estimates on modulation spaces for Schr\"{o}dinger evolution operators $e^{itH}$ are performed in the works \cite{kki1,kki2,kki3,kki4}. In particular, the most general result is given in \cite{kki4}, where the authors  study the operator $H=\Delta-V(t,x)$, the time-dependent potential $V(t,x)$ being smooth and quadratic or sub-quadratic. They obtain boundedness results for the propagator $e^{it H}$ in the unweighted modulation spaces $M^{p,q}$, $1\leq p,q\leq\infty$ (see definition in Section $2$).  In these papers the key idea is to consider, as window function to estimate the modulation space norm of the solution $e^{itH}u_0$, the Schwartz function $\varphi(t,\cdot)= e^{i t \Delta}\varphi$, for a given $\varphi=\varphi(0,\cdot)\in\cS(\rd)\setminus\{0\}$, used to estimate the modulation norm of the initial datum $u_0$.\par
Related results in the $L^p$-theory are contained in~\cite{js94,js95}.
Finally, more general Hamiltonians and potentials are studied in \cite{MetapWiener13,nuovareference}. In particular the perturbation operators have symbols in the so-called Sj\"ostrand class $S_w=M^{\infty,1}(\rdd).$\par
In what follows we shall present to the reader a proof, new in literature, of the boundedness of $e^{it H}$ on modulation spaces when the perturbation $\sigma^w$ belongs to the following
 scale of  modulation spaces
\begin{equation}\label{cinque}
S^s_w=M^{\infty}_{1\otimes v_s}(\rdd),\quad v_s(z)=\langle
z\rangle ^s=(1+|z|^2)^{s/2}, \,\,z\in\rdd,
\end{equation}
with  the  parameter $s>2d$ (See Theorem \ref{teofinal} in the sequel). Observe that
\begin{equation}\label{inteSs}
\bigcap_{s\geq 0}S^s_w=S^0_{0,0},
\end{equation}
where $S^0_{0,0}$ is the   H\"ormander class of all $\sigma\in\cC^\infty(\rdd)$ satisfying
\begin{equation}\label{HCS0}
|\partial^\a \sigma(z)|\leq C_\a, \quad \a\in\bZ^{2d}_+,
\,\,z=(x,\xi)\in\rdd.
\end{equation}
For $s\to 2d+$, the symbols in $S^s_w$ have a smaller
regularity.  In particular, for
$s>2d$, $S^s_w\subset \cC^0(\rdd)$,  but the differentiability is lost
in general as soon as $s\leq 2d +1$. \par
Finally, note that $ \bigcup_{s>2d} S^s_w\subset S_w\subset \cC^0(\rdd)$, with strict inclusions.\par
\medskip
The representation of the evolution operator $e^{ita^w}$ which is related to the unperturbed Schr\"odinger equation $\sigma^w=0$ is already well understood since  $e^{ita^w}$ is then a metaplectic operator. To benefit of non-expert readers, we shall study in  detail this case in the next Chapter $5$. We recall that
metaplectic operators quantize linear symplectic transformations of the phase-space and  arise as intertwining  operators of the Schr\"odinger
representation of the Heisenberg group.   \par
If we consider the perturbed problem  ($\sigma^w\not=0$), a natural question is as follows:\par
\smallskip
\emph{What is the representation  of  $e^{it H}$ in presence of perturbations?}\par
\smallskip
An answer in the case $\sigma^w=\sigma_t(x)$, that is, the potential is a multiplication by a smooth function $\sigma_t\in\cC^\infty(\rd)$, satisfying additional decay properties together with its derivatives, is contained in  Weinstein~\cite{Wein85}. There the evolution $e^{itH}$ was written as a composition of a metaplectic operator with a pseudodifferential operator having symbol in particular subclasses of
H{\"o}rmander classes. We generalize widely the latter result by considering rough potentials with symbols in $ S^s_w$ and showing that $e^{itH}$ is the product of a metaplectic operator and of a pseudodifferential operator (with symbol in $S^s_w$) for every $t\in\bR$, that is, as defined below,  a generalized metaplectic operator.
The class of the generalized metaplectic operators enters the class of the Fourier integral operators introduced and studied in \cite{Wiener}.  \par
First, we recall
the classical metaplectic operators.
Let $\A $ be a symplectic matrix on $\rdd $ (we write $\A\in Sp(d,\bR))$, i.e., $\A$ is a $2d\times 2d$ invertible matrix such that $\A ^T J \A  = J$,
where
\begin{equation}
\label{matriceJ}
J=\begin{pmatrix} 0&I_d\\-I_d&0\end{pmatrix}
\end{equation}
is related to the standard symplectic form on $\rdd $
\begin{equation}
\omega(x,y)=\;^t\!xJy, \qquad x,y\in\R^{2d}.
\label{symp}\end{equation}
Then, the
metaplectic operator $\mu (\cA )$ may be defined by the intertwining
relation
\begin{equation}\label{metap}
\pi (\cA z) = c_\cA  \, \mu (\cA ) \pi (z) \mu (\cA )\inv  \quad  \forall
z\in \rdd \, ,
\end{equation}
with a phase factor $c_\cA \in \bC , |c_{\cA } | =1$  (for details, see e.g. \cite{folland89,Gos11}).
 The kernel or so-called  Gabor matrix of the
metaplectic operator with respect to the set of the \tfs s $\pi (z)$ satisfies the following estimate. If $\A \in \mathrm{Sp}(d,\R)$ and $g\in \cS (\rd
)$, then for every $N\geq 0$ there exists a $C_N >0$ such that
\begin{equation}
  \label{eq:kh21}
|\langle \mu (\A ) \pi (z)g, \pi (w) g\rangle |\leq C_N\langle w-\A z
\rangle ^{-N},\qquad w,z\in\R^{2d}.
\end{equation}
The Gabor matrix is the point of departure in the definition of the generalized metaplectic operators.
\begin{definition}\label{def1.1} Given $\cA \in \Spnr $,
  $g\in\cS(\rd)$,  and $s\geq0$, we say that  a
 linear operator $T:\cS(\rd)\to\cS'(\rd)$ is  a
 generalized metaplectic operator, in short  $T\in FIO(\cA,s)$, if the  Gabor matrix of $T$ satisfies the decay
condition
 \begin{equation}\label{asterisco}
|\langle T \pi(z) g,\pi(w)g\rangle|\leq {C}\langle w-\cA z\rangle^{-s},\qquad  w,z\in \rdd.
\end{equation}
\end{definition}
The union
\[
FIO(Sp(d,\R),s)=\bigcup_{\mathcal{A}\in Sp(d,\R)} FIO(\mathcal{A},s)
\]
is then called the  class of \emph{generalized metaplectic operators}. A generalized metaplectic operator is then proved to be the composition of a metaplectic operator with a pseudodifferential operator having symbol in the class $S^s_w$, cf. Theorem \ref{pseudomu} below.
\vskip0.3truecm
The main result, that is Theorem \ref{teofinal} in the sequel, shows that the propagator of the   perturbed problem
\eqref{C1} -- \eqref{C1bis}  is  a generalized metaplectic operator.

\smallskip
The last problem which arises in the study of the evolution $e^{itH}$ is related to propagation of singularities.\par
\smallskip
\emph{How the evolution $e^{itH}$ propagates the singularities of the initial datum $u_0$?}\par
\smallskip

To answer this question, we employ  a new definition of wave front set, called Gabor wave front set, which is a generalization of the global wave front set introduced by H\"{o}rmander in 1991 \cite{hormanderglobalwfs91} and redefined by using time-frequency analysis in \cite{RWwavefrontset}. See also \cite{wavefrontsetshubin13}. The definition of the classical, the global and the Gabor wave front set, together with a comparison with related results in the literature, is detailed in the last section.\par
\vskip0.1truecm
The contents of the next sections are the following. In Section $2$ we recall the main time-frequency analysis tools and properties of pseudodifferential operators useful for our results. In Section $3$ we survey the main properties of generalized metaplectic operators. Section $4$ and $5$ contain the study of the unperturbed  and perturbed Schr\"odinger equation, respectively. In particular, in Section $5$ it is stated and proved the main result of this paper (Theorem \ref{teofinal}).
Finally, to give a whole treatment of this topic, in  Section $6$ we recall  the definition of the Gabor wave front set and review results of propagation of singularities of the evolution $e^{itH}$, showing some important examples.

\vskip0.3truecm
\section{Preliminaries and  \tfa \,tools}
 We refer the reader to  \cite{book} for an introduction to time-frequency concepts. We write $xy=x\cdot y$ for  the scalar product on
 $\Ren$ and $|t|^2=t\cdot t$ for $t,x,y \in\Ren$.
 The Schwartz class is denoted by
 $\sch(\Ren)$, the space of tempered
 distributions by  $\sch'(\Ren)$. The brackets  $\la f,g\ra$
 denote the extension to $\sch '
 (\Ren)\times\sch (\Ren)$ of the inner
 product $\la f,g\ra=\int f(t){\overline
 {g(t)}}dt$ on $L^2(\Ren)$. The Fourier
 transform is given by ${\hat
   {f}}(\o)=\Fur f(\o)=\int
 f(t)e^{-2\pi i t\o}dt$.
  We  write
 $A \asymp B$  for the equivalence  $c^{-1}B\leq
 A\leq c B$.
\subsection{Modulation  spaces}
Consider a distribution $f\in\cS '(\rd)$
and a Schwartz function $g\in\cS(\rd)\setminus\{0\}$ (the so-called
{\it window}).
The short-time Fourier transform (STFT) of $f$ with respect to $g$ is defined by \eqref{STFT}.
 The  \stft\ is well-defined whenever  the bracket $\langle \cdot , \cdot \rangle$ makes sense for
dual pairs of function or (ultra-)distribution spaces, in particular for $f\in
\cS ' (\rd )$ and $g\in \cS (\rd )$,  or for $f,g\in\lrd$.

Weighted modulation spaces measure the decay of the STFT on the time-frequency (or phase space) plane and were defined by Feichtinger in the 80's \cite{F1}.\par

Let us first introduce the weight functions.  A weight function $v$ on $\rdd$ is submultiplicative if $ v(z_1+z_2)\leq v(z_1)v(z_2)$, for all $z_1,z_2\in\Renn.$  We shall work with the weight functions
\begin{equation} v_s(z)=\la z\ra^s=(1+|z|^2)^{\frac s 2},\quad s\in\R,
\end{equation}
which are submultiplicative for $s\geq0$.\par
If $\cA\in \mathrm{GL}(d,\bR )$, the class of real $d\times d$ invertible matrices, then  $|\cA z|$ defines an equivalent
norm on $\rdd$, hence for every $s\in\R$, there exist $C_1,C_2>0$ such
that
\begin{equation}\label{pesieq}
C_1 v_s(z)\leq v_s(\cA z)\leq C_2 v_s(z),\quad \forall z\in\rdd.
\end{equation}
 For $s\geq0$, we denote by $\mathcal{M}_{v_s}(\rdd)$ the space of $v_s$-moderate weights on $\rdd$;
these  are measurable positive functions $m$ satisfying $m(z+\zeta)\leq C
v_s(z)m(\zeta)$ for every $z,\zeta\in\rdd$.

\begin{definition}  \label{prva}
Given  $g\in\cS(\rd)$, $s\geq0$, a  weight
function $m\in\mathcal{M}_{v_s}(\rdd)$, and $1\leq p,q\leq
\infty$, the {\it
  modulation space} $M^{p,q}_m(\Ren)$ consists of all tempered
distributions $f\in \cS' (\rd) $ such that $V_gf\in L^{p,q}_m(\Renn )$
(weighted mixed-norm spaces). The norm on $M^{p,q}_m(\rd)$ is
\begin{equation}\label{defmod}
\|f\|_{M^{p,q}_m}=\|V_gf\|_{L^{p,q}_m}=\left(\int_{\Ren}
  \left(\int_{\Ren}|V_gf(x,\o)|^pm(x,\o)^p\,
    dx\right)^{q/p}d\o\right)^{1/q}  \,
\end{equation}
(with obvious modifications for $p=\infty$ or $q=\infty$).
\end{definition}
 When $p=q$, we simply write $M^{p}_m(\rd)$ instead of
 $M^{p,p}_m(\rd)$. The spaces $M^{p,q}_m(\rd)$ are Banach spaces,  and
 every nonzero $g\in M^{1}_{v_s}(\rd)$ yields an equivalent norm in
 \eqref{defmod}. Thus  $M^{p,q}_m(\Ren)$ is independent of the choice
 of $g\in  M^{1}_{v_s}(\rd)$.
\par In the sequel we shall use the following inversion formula for
the STFT (see  (\cite[Proposition 11.3.2]{book}): assume $g\in M^{1}_v(\rd)\setminus\{0\}$,
 $f\in M^{p,q}_m(\rd)$, then
\begin{equation}\label{invformula}
f=\frac1{\|g\|_2^2}\int_{\R^{2d}} V_g f(z) \pi (z)  g\, dz \, ,
\end{equation}
and the  equality holds in $M^{p,q}_m(\rd)$.\par
 The adjoint operator of $V_g$,  defined by
 $$V_g^\ast F(t)=\intrdd F(z)  \pi (z) g dz \, ,
 $$
 maps the Banach space $L^{p,q}_m(\rdd)$ into $M^{p,q}_m(\rd)$. In particular, if $F=V_g f$ the inversion formula \eqref{invformula} becomes
 \begin{equation}\label{treduetre}
 {\rm Id}_{M^{p,q}_m}=\frac 1 {\|g\|_2^2} V_g^\ast V_g.
 \end{equation}

\subsection{Metaplectic and Pseudodifferential Operators}

Consider $g\in\cS(\rd)\setminus\{0\}$ with $\|g\|_2=1$. Then, using the inversion formula \eqref{invformula}, any
linear continuous operator $T:\cS(\rd)\to \cS'(\rd)$ admits the
following time-frequency representation:
\begin{equation}\label{KGbM}
V_g(Tf)(w)=\intrdd k(w,z) V_g f(z) \,dz.
\end{equation}
where we call the kernel
\begin{equation}\label{GbM}
k(w,z):=\la T\pi(z)g,\pi(w) g\ra,\quad w,z\in\rdd
\end{equation}
the continuous Gabor matrix of the operator $T$. The name is related to the discretization of $T$ by means of  Gabor frames. More precisely, the collection of
time-frequency shifts $\G(g,\Lambda)=\{\pi(\lambda)g:\
\lambda\in\Lambda\}$ for a  non-zero $g\in \cS(\rd)$ (or more generally  $g\in L^2(\rd)$) is called a Gabor frame, if there exist
constants $A,B>0$ such that
\begin{equation}\label{gaborframe}
A\|f\|_2^2\leq\sum_{\lambda\in\Lambda}|\langle f,\pi(\lambda)g\rangle|^2\leq B\|f\|^2_2\qquad \forall f\in L^2(\rd).
\end{equation}
This implies the expansion with unconditional convergence in $L^2(\rd)$:
\begin{equation}\label{parsevalframe}
f=\sum_{\lambda\in\Lambda}\langle f,\pi(\lambda)\gamma\rangle\pi(\lambda)g\qquad\forall f\in L^2(\rd)
\end{equation}
where $\gamma$ is a dual window of $g$.
Using \eqref{parsevalframe} the following Gabor decomposition of the operator $T$ is obtained:
$$Tf(x)=\sum_{\mu\in\Lambda}\sum_{\lambda\in\Lambda}\underbrace{\la T \pi(\lambda)g,\pi(\mu)g\ra}_{T_{\mu \,\lambda}} c_\lambda  \pi(\mu)\gamma,
$$
where $c_\lambda =\la f,\pi(\lambda)\gamma\ra$ are the Gabor coefficients of $f$ with respect to the dual Gabor frame $\G(\gamma,\Lambda)$ and the infinite matrix $\{T_{\mu \,\lambda}\}_{\mu, \,\lambda\in\Lambda}$ is called the \emph {Gabor matrix} of $T$. For  simplicity, from now on we shall present the theory only from a continuous point of view, but the discrete counterpart by Gabor frames works as well, and this is indeed the starting point for further  numerical implementations. \par
\bigskip
{\bf Metaplectic Operators.} Given a symplectic matrix $\cA \in Sp(d,\R)$, the corresponding
metaplectic operator $\mu (\cA )$ can be defined by the intertwining relation \eqref{metap} (see also Section $4$).\par
For matrices $\A\in Sp(d,\bR)$ in special form, the corresponding metaplectic operators can be computed explicitly. Precisely, For
$f\in L^2(\R^d)$, we have
\begin{align}
\mu\left(\begin{pmatrix} A&0\\ 0&\;^t\!A^{-1}\end{pmatrix}\right)f(x)
&=(\det A)^{-1/2}f(A^{-1}x)\label{diag}\\
\mu\left(\begin{pmatrix} I&0\\ C&I\end{pmatrix}\right)f(x)
&=\pm e^{-i\pi Cx \cdot x}f(x)\label{lower}\\
\mu\left(J\right)&=i^{d/2}\mc{F}^{-1}\label{iot},
\end{align}
where $\mc{F}$ denotes the Fourier transform.

The Gabor matrix  of a metaplectic operator $ \mu (\cA ) $ is concentrated along the graph of the symplectic phase-space transformation $\cA$ and decays super polynomially outside, as expressed in the first part of the following version of \cite[Lemma 2.2]{MetapWiener13}. The second part gives a technical information used later (the proof is analogous to the one of \cite[Lemma 2.2]{MetapWiener13}).

\begin{lemma} \label{lkh1}
  (i) Fix $g\in \cS (\rd )$ and $\A \in Sp(d,\R)$, then, for all $N\geq
  0$,
  \begin{equation}
    \label{eq:kh1}
    |\langle \mu (\A )\pi (z)g, \pi (w)g\rangle | \leq C_N \langle
    w-\A z\rangle ^{-N}  \, .
  \end{equation}
(ii) If $\sigma \in M^\infty _{1\otimes v_s}$ and $\A \in Sp(d,\R)$,
then $\sigma \circ \A \in M^\infty _{1\otimes v_s}$ and
\begin{equation}
  \label{eq:kh2}
\|\sigma \circ \A \inv \|_{M^\infty _{1\otimes v_s}} \leq   \|(\A ^T)\inv
\|^s \, \|V_{\Phi
  \circ \A  } \Phi \|_{L^1_{v_s} } \|\sigma   \|_{M^\infty _{1\otimes v_s}},
\end{equation}
where $\Phi\in\cS(\rdd)$ is the window used to compute the norms of $\sigma$ and $\sigma \circ \A \inv $.
\end{lemma}
\medskip
In particular, using \eqref{KGbM}
 for $T=\mu(\A)$ and the estimate \eqref{eq:kh1}, we can write for every $N\geq 0$
 \begin{equation*}
 |V_g(\mu(\A)f)(w)|\leq C_N \intrdd \la w-\A z\ra^{-N} |V_g f(z)| \,dz.
 \end{equation*}
Consider now a weight $m\in\mathcal{M}_{v_s}$, $s\geq 0$. Then the $v_s$-moderateness yields
\begin{equation*}
 |V_g(\mu(\A)f)(w)| m(w)\leq C_N \intrdd  \la w-\A z\ra^{s-N} m(\A z)|V_g f(z)| \,dz
 \end{equation*}
 and choosing $N$ such that $s-N<-2d$ by Young's inequality we obtain the following boundedness result:
 \begin{theorem}
 Consider $m\in\mathcal{M}_{v_s}$, $s\geq 0$ and $\A \in Sp(d,\R)$. Then the metaplectic operator $\mu(\A)$ is bounded from $M^{p}_{m\circ \cA}(\rd)$ into  $M^{p}_{m}(\rd)$, $1\leq p\leq \infty$, with the norm estimate
 \begin{equation}\label{contmua}
 \|\mu(\A)f\|_{M^{p}_m}\leq C_s \|f\|_{M^{p}_{m\circ \cA}}.
 \end{equation}
 \end{theorem}
 The continuity property of a metaplectic operator $\mu(\A)$ on $M^{p,q}(\rd)$, with $p\not=q$, fails in general. Indeed,
 an example is provided by the multiplication operator defined in \eqref{lower}, which is bounded on $M^{p,q}(\rd)$ if and only if $p=q$, as proved in \cite[Proposition 7.1]{fio1}.

\bigskip

{\bf  Pseudodifferential Operators.}
These operators can be described by the off-diagonal decay of their corresponding  Gabor matrices, if their symbols are chosen in suitable modulations spaces. This is the main insight of the papers \cite{charly06,GR}. Namely, we have:
\begin{proposition}[\cite{charly06}]
  \label{charpsdo}
   Fix $g \in\cS(\rd)$, consider $\sigma \in \cS'(\rdd )$ and $s\in\bR$.\\
(i) The symbol  $\sigma$ is in $M^{\infty,1} _{1\otimes v_s}(\rdd )$ \fif\ there exists a function $H\in  L^1_{v_s}(\rdd )$ such that
\begin{equation}
  \label{eq:kh9}
  |\langle \sigma ^w \pi (z) g , \pi (w) g\rangle | \leq H(w-z) \qquad
  \forall w,z \in \rdd \, .
\end{equation}
The function $H$ can be chosen as
\begin{equation}
  \label{eq:kh12}
H(z)=\sup_{u\in\rdd}|V_{\Phi} \sigma (u,j(z))|,
\end{equation}
where $j(z)=(z_2,-z_1)$ for $z=(z_1,z_2)\in\rdd$ and the  window
function $\Phi=W(g,g)$ is  the Wigner distribution of $g$.\\
(ii)  We have $\sigma \in
 M^{\infty}_{1\otimes v_s}(\rdd )$ \fif\
 \begin{equation}
   \label{eq:kh9s}
   |\langle \sigma^w  \pi (z) g , \pi (w) g\rangle | \leq C \la w-z\ra^{-s} \qquad
   \forall w,z \in \rdd,
 \end{equation}
 where  the constant $C$ in \eqref{eq:kh9s} satisfies
   \begin{equation}\label{cost}
   C\asymp \|\sigma\|_{M^{\infty}_{1\otimes v_s}}.
   \end{equation}
\end{proposition}

 In the context of Definition~\ref{def1.1} the class of pseudodifferential
operators with a symbol in $M^\infty_{1\otimes v_s} (\rdd )$ is just $FIO (\mathrm{Id},
s)$.

The previous characterization is the key idea to show that
the preceding pseudodifferential operators give rise to  a sub-algebra of $\mathcal{B}(L^2(\rd))$ (the algebra of linear bounded operators on $L^2(\rd)$) which is inverse-closed (or enjoys the so-called Wiener property), as contained in the following results \cite{charly06,GR,wiener31}.
\begin{theorem}\label{fund}
Assume that $\sigma \in  M^{\infty,1} _{1\otimes v_s}(\rdd ) $ (resp. $\sigma \in M^\infty _{1\otimes v_s}(\rdd )$, with $s>2d$). Then:

(i) \emph{Boundedness:} $\sigma ^w$ is bounded on every \modsp\ $M^{p,q}_m(\rd )$ for
$1\leq p,q \leq \infty   $ and every $v_s$-moderate weight $m$.

(ii) \emph{Algebra property:} If $\sigma _1 , \sigma _2 \in  M^{\infty,1} _{1\otimes v_s}(\rdd ) $ (resp. $\sigma \in M^\infty _{1\otimes v_s}(\rdd )$),
then  $\sigma _1^w \sigma _2^w = \tau ^w$ with a symbol $\tau \in
M^{\infty,1} _{1\otimes v_s}(\rdd ) $ (resp. $\sigma \in M^\infty _{1\otimes v_s}(\rdd )$).

(iii) \emph{Wiener property:} If $\sigma ^w $ is invertible on $\lrd
$, then $(\sigma ^w)\inv = \tau ^w$  with a symbol $\tau \in
M^{\infty,1} _{1\otimes v_s}(\rdd ) $ (resp. $\sigma \in M^\infty _{1\otimes v_s}(\rdd )$).
\end{theorem}
The algebra property represents one of the main ingredients in the study of the properties of the Schr\" odinger propagator $e^{i t H}$ developed in Section $5$.  In particular, the kernel of the composition of $n$ pseudodifferential operators $\sigma_j^w$, with $\sigma _j \in M^\infty_{1\otimes v_s}(\rdd)$, $j=1, \dots ,n $, and $s>2d$, satisfies
  \begin{equation}
    \label{eq:kh3}
    |\langle \sigma _1^w \sigma _2 ^w \dots \sigma _n^w \pi (z)g, \pi
    (w)g\rangle | \leq C_0 C_1\dots C_n\la  w-z\ra^{-s} \, \quad
    w,z\in \rdd,
  \end{equation}
  where \begin{equation}\label{costj}
    C_j\asymp \|\sigma_j\|_{M^{\infty}_{1\otimes v_s}}
    \end{equation}
and $C_0>0$ depends only on $s$.
The proof is straightforward: in fact, the characterization of Proposition~\ref{charpsdo}  says that $\sigma _j \in M^\infty_{1\otimes v_s} (\rdd )$ if and only if
$   |\langle \sigma _j^w \pi (z)g, \pi
    (w)g\rangle | \leq C_j \la  w-z\ra^{-s}  $ and the
    result follows from  induction and the fact that, for $s>2d$, the weights $v_s$ are subconvolutive: $v_s^{-1}\ast v_s^{-1}\leq C_0\, v_s^{-1}$ (cf. \cite[Lemma 11.1.1]{book}).

\section{Properties of the class $FIO(\A,s)$}
Generalized metaplectic operators solve evolution equations of the form \eqref{C1}-\eqref{C1bis}, when the perturbation is a pseudodifferential operator with symbol in the classes $S^s_w$.
They were introduced and studied in \cite{Wiener}, as classes of Fourier integral operators (FIOs) associated with  linear symplectic transformations of the phase-space. We observe that this work studies also FIOs associated to more general symplectomorphisms.  \par
As already mentioned, the crucial property of the generalized metaplectic operators we shall need in the study of Schr\" odinger equations is the algebra property of the class $FIO(\A,s)$.
However, for sake of completeness, we recall the main properties of this class, obtained in~\cite{Wiener}, and also other properties
 peculiar of the weighted versions of the Sj{\"o}strand class, i.e., the modulation spaces $M^{\infty,1}_{1\otimes v_s}(\rdd)$, cf. \cite{MetapWiener13}.

First of all,  the definition of the class $FIO(\cA ,s)$ is independent of the  window function $g\in\cS(\rd)\setminus\{0\}$ chosen in the Definition \ref{def1.1}. The class $FIO(\cA ,s)$ is nonempty: every classical metaplectic operator $\mu(\A)$ satisfies \eqref{eq:kh21} for every $s\geq 0$ and so $\mu(\A)\in FIO(\cA ,s)$ for every $s\geq 0$. \par
Using the estimate of the Gabor matrix  for a generalized metaplectic operator in \eqref{eq:kh21} and repeating similar  arguments as for the classical metaplectic operators $\mu(\A)$ in the previous section, we obtain the following boundedness results for the class $FIO(\cA ,s)$:
\begin{theorem}\label{T31}
Fix  $\A\in\Spnr$, $s>2d$, $m\in\cM_{v_s}$ and $T\in FIO(\A,s)$. Then $T$ extends to a bounded operator from $M^p_{m\circ\A}(\rd)$ to $M^p_{m}(\rd)$, $1\leq p\leq
\infty$.\end{theorem}
In particular, since the weights $v_s$ and $v_s\circ \A$ are equivalent for every $s\in\bR$ and $\A\in Sp(d,\bR)$,  the generalized metaplectic operator $T$ in Theorem \ref{T31} is bounded on $M^p_{v_s}(\rd )$ and on $M^p_{1/v_s}(\rd )$.
The previous theorem for $p=2$ and $m\equiv 1$ says that the classes $FIO(\A,s)$ are subclasses of $\mathcal{B} (\lrd )$. The next results show that their union $FIO(\Spnr,s)$ is indeed an algebra in $\mathcal{B} (\lrd )$ which enjoys the property of inverse-closedness:
\begin{theorem}\label{prod1}
We have:\\
(i) If $T^{(i)}\in FIO(\A_i,s_i)$,  $\A_i\in\Spnr$,  with $s_i>2d$, $i=1,2$; then  $T^{(1)}T^{(2)}\in FIO(\A_1\circ \A_2, s)$
with $s=\min(s_1,s_2)$. Consequently,  the class $FIO(\Spnr,s)$ is an algebra with respect to the composition of operators.\\
(ii) If $T\in FIO(\A,s)$,  $s>2d$, and   $T$ is invertible on $L^2(\rd)$, then $T^{-1} \in FIO(\A^{-1},s)$. Consequently, the algebra $FIO(\Spnr,s)$ is inverse-closed in $\mathcal{B} (\lrd )$.
\end{theorem}
Finally, the operators in the classes $FIO(\A,s)$ are not \emph{abstract ghosts} since they can be explicitly written as a simple composition of a classical metaplectic operator and a pseudodifferential operator, as expressed in \cite[Theorem 5.4]{Wiener}, recalled below.
\begin{theorem}\label{pseudomu}
Fix $s>2d$ and $\cA \in \Spnr $.  A linear continuous operator $T:
\cS(\rd)\to\cS'(\rd)$ is in   $FIO(\A,s)$ if and only if there
exist symbols $\sigma_1, \sigma_2 \in
M^{\infty}_{1\otimes v_s}(\rdd)$, 
such that
\begin{equation}\label{pseudomu1}
T=\sigma_1^w(x,D)\mu(\A)\quad \mbox{and}\quad
T=\mu(\A)\sigma^w_2(x,D).
\end{equation}
The symbols $\sigma _1$ and $\sigma _2$ are related by
\begin{equation}\label{hormander}\sigma_2=\sigma_1\circ\A.\end{equation}
\end{theorem}

 The characterization \eqref{pseudomu1} works also for other $\tau$-forms of pseudodifferential operators $a_\tau(x,D)$, $\tau\in [0,1]$, where the $\tau$-quantization of a symbol $a(x,\xi)$ on the phase-space is formally defined by
 $$a\mapsto a_\tau(x,D) f=\iint_{\rdd} e^{2\pi i(x-y)\xi} a(\tau x+(1-\tau)y,\xi) f(y) dy\, d\xi$$
($\tau=1/2$ is the Weyl quatization whereas $\tau=1$ is the Kohn-Nirenberg correspondence). Instead \eqref{hormander} is peculiar of the Weyl correspondence and is a consequence of  the symplectic invariance property of the Weyl calculus (e.g. \cite[Theorem 18.5.9]{hormander3}):
$$ \mu (A)\inv \sigma^w\mu(\A) = \mu (\A ) (\sigma \circ \A ) ^w. $$

For $T\in FIO(\A,s)$ with  $\A=\begin{pmatrix} A&B\\C&D\end{pmatrix}\in
Sp(d,\R)$ satisfying the additional condition $ \det A\not=0$ we have the following characterization, that can be proved by using the same arguments as those in the result \cite[Theorem 5.1]{MetapWiener13} (related to symbols in the classes $M^{\infty,1}_{1\otimes v_s}(\rdd)$). We leave the details to the interested reader.
\begin{theorem}
$T\in
FIO(\mathcal{A},s)$ if and only if $T$ has the following integral representation
\begin{equation}\label{fiotipo1}
Tf(x)=\int_{\rd} e^{2\pi i \Phi\phas} \sigma\phas \hat{f}(\xi)d\xi
\end{equation}
with the  phase  $ \Phi(x,\xi)=\frac12  x CA^{-1}x+
\xi  A^{-1} x-\frac12\xi  A^{-1}B\xi$ and a  symbol $\sigma\in M^{\infty}_{1\otimes v_s}(\rdd)$.
\end{theorem}
An operator $T$ in the form \eqref{fiotipo1} is called a type I Fourier integral operator, with phase $\Phi$ and symbol $\sigma$.
If we drop the condition $\det A\not=0$ the operator $T$ is no more a FIO of type I, but certainly can be characterized by other suitable integral representations.   Integral representations for classical  metaplectic operators were studied by Morsche and Oonincx in \cite{MO2002}. Our object of further investigation will be to study the  compositions of the integral representations  in \cite{MO2002} with pseudodifferential operators and derive an integral expression for any $T\in FIO(\A,s)$.

\section{Metaplectic Operators as solutions of the unperturbed problem}

We first give a brief review concerning the solution of
\begin{equation}\label{C12}
\begin{cases} i \displaystyle\frac{\partial
u}{\partial t} +a^w u=0\\
u(0,x)=u_0(x),
\end{cases}
\end{equation}
when the symbol $a(x,\xi)$ is a quadratic form on the phase space. Writing $u(t,\cdot)=M_t u(0,\cdot)$, we obtain the following equation for the solution  operator $M_t$:
\begin{equation}\label{C12Mt}
\begin{cases} i \displaystyle\frac{d
M_t}{d t} +a^w M_t=0\\
M_0=I.
\end{cases}
\end{equation}
The properties of the operator $M_t$ are well-known an can be found in different textbooks. We refer to \cite{Gos11,folland89,taylor} and references therein. Moreover, we address to Voros \cite[Section 4.6]{voros} (see also  \cite[Section 3]{Wein85}) for a short survey of the subject,  from the viewpoint of group theory and geometrical quantization.

Let us denote by $\mathcal{P}_2$ the set of  real-valued polynomial
of degree $\leq 2$ on $\rdd$. This set is a Lie algebra with respect to the Poisson bracket:
$$ \{f,g\}=\sum_{j=1}^{d}\frac{\partial
f}{\partial x_j}\frac{\partial
g}{\partial \xi_j}-\frac{\partial
f}{\partial \xi_j}\frac{\partial
g}{\partial x_j},\quad f,g\in \mathcal{P}_2.
$$
Using the Weyl correspondence, to any $f,g\in \mathcal{P}_2$ correspond  the Weyl
operators $f^w, g^w$ with the following commutation relation
$$2\pi i[f^w, g^w]=- \{f,g\}^w,$$
so the mapping $f\mapsto if^w$ is an isomorphism between the Lie algebra $(\mathcal{P}_2,\{ \})$ and a Lie algebra of essentially skew-adjoint operators (observe that $f^w$ is essentially self-adjoint because $f$ is real-valued), spanned by the constant and the linear symplectic vector fields. Let us now restrict our attention to the  subalgebra $hp_2\subset \mathcal{P}_2$ of second order homogeneous polynomials on $\rdd$.
Consider the symplectic Lie algebra $ \spdr$ of $2d\times 2d$ matrices such that $J {}^t \mathbb{A}+ \mathbb{A}J=0$. Any given matrix
\begin{equation}
\label{matricesimp}
\mathbb{A}=\begin{pmatrix}
A&B\\C&-{}^t A\end{pmatrix}\in \spdr,
\end{equation}
with $A,B, C\in M(d,\bR)$, $B,C$ symmetric,  defines a quadratic form $\mathcal{P}_\mathbb{A}(x,\xi)$ in $\rdd$ via
the formula
\[
\mathcal{P}_\mathbb{A}(x,\xi)=-\frac{1}{2}{}^t(x,\xi)J\mathbb{A} (x,\xi)=\frac{1}{2}\xi\cdot
B\xi+\xi\cdot A x-\frac{1}{2}x\cdot Cx.
\]
The mapping $\mathcal{P}: \spdr \to hp_2$, such that $\mathbb{A}\mapsto \mathcal{P}_\mathbb{A}$, is an isomorphism between  the Lie algebra $(\spdr, [ \, ] )$ and the Lie subalgebra $(hp_2, \{  \})$.

From the Weyl quantization, the
quadratic polynomial $P_\mathbb{A}$ corresponds to the Weyl
operator
 $\mathcal{P}_\mathbb{A}^w$
 defined by
\[
\mathcal{P}_\mathbb{A}^w=-\frac{1}{8\pi^2}\sum_{j,k=1}^d
B_{j,k}\frac{\partial^2}{\partial
x_j\partial x_k}-\frac{i}{2\pi}\sum_{j,k=1}^d
A_{j,k}
x_j\frac{\partial}{\partial{x_k}}-
\frac{i}{4\pi}{\rm
Tr}(A)-\frac12 \sum_{j,k=1}^d C_{j,k}
x_jx_k.
\]
We define a representation $d\mu$ of the Lie algebra $\spdr$ by
$$ d\mu: \mathbb{A}\mapsto  i \mathcal{P}_\mathbb{A}^w.$$
The operator $i\mathcal{P}_\mathbb{A}^w$ is called the infinitesimal metaplectic operator corresponding to $\mathbb{A}$, or, equivalently, to $\mathcal{P}_\mathbb{A}$.  Hence the infinitesimal metaplectic operators form a Lie algebra of essentially skew-adjoint operators.

Next, general results about analytic vectors \cite[Appendix D]{taylor} imply that the representation $d\mu$  generates a unitary representation
$$\mu:\, \widetilde{Sp(d,\bR)}\to U(L^2(\rd))
$$
of the universal covering group $\widetilde{Sp(d,\bR)}$ of $Sp(d,\bR)$. Actually, $\mu$ can be exponentiated to a representation of the two-fold cover of $Sp(d,\bR)$, commonly called the metaplectic group and denoted by $Mp(d,\bR)$. The corresponding group of operators in $U(L^2(\rd))$ are called metaplectic operators. Such operators  preserve $\cS(\rd)$ and extend to $\cS'(\rd)$.

Suppose we are given a vector $i a^w$ with $a\in hp_2$ (that is, an infinitesimal metaplectic operator), then the Schr\"{o}dinger equation $d M_t/ dt= i a^w M_t$, ($M_0=I$), that is $\eqref{C12Mt}$, has a unique solution $M_t$ in $Mp(d,\bR)$ (or equivalently, there exists a unique metaplectic operator  $\mu(M_t)$ solution of the previous equation). The corresponding curve $\A_t=\pi(M_t)$, in $Sp(d,\bR)$ (where $\pi$  is the projection of $Mp(d,\bR)$ onto $Sp(d,\bR)$), describes the motion of the classical system with the hamiltonian $a^w$. By abuse of notation, since the ambiguity is only a matter of sign, from now onward we simply write $\mu(\A_t)$ in place of $\mu(M_t)$,  considering the metaplectic representation  $\mu\, : Sp(d,\bR) \to  U(L^2(\rd))$.
Hence, the solution to the Cauchy problem \eqref{C12}, with $a^w=\mathcal{P}^w_{\mathbb{A}}$ and $\mathbb{A}$ in \eqref{matricesimp} can be written as
$$u=e^{ita^w}u_0=\mu(\mathcal{A}_t)u_0$$
with $\mathcal{A}_t=e^{t\mathbb{A}}\in Sp(d,\R)$. \par
The previous theory has further extensions. First, we note that the Schr\"{o}dinger representation $\rho$ of the Heisenberg group  $\H^d$
 and $\mu$ fit together to give rise to  the \emph{extended metaplectic representation} $\mu_e$ (see \cite{AEFK1, AEF10} and references therein). We briefly review its
 construction.
 \par
 The Heisenberg group $\H^d$ is the group obtained by defining on
 $\bR^{2d+1}$  the product$$
 (z,t)\cdot(z',t')=(z+z',t+t'+\frac{1}{2}\omega(z,z')),
 \quad z,z'\in\rdd,\,\,t,t'\in\bR$$
 where $\omega$ stands for the standard symplectic form in
 $\R^{2d}$ given in \eqref{symp}.
 The Schr\"odinger representation of the group $\H^d$
 on $\lrd$ is then defined by
 $$
 \rho(x,\xi,t)f(y)=e^{2\pi it}e^{-\pi i x \xi}
 e^{2\pi i \xi y}f(y-x) = e^{2\pi it}e^{-\pi i
 x \xi}  M_\xi T_x f(y)=e^{2\pi it}e^{-\pi i
  x \xi} \pi(z)f(y),
 $$
 where $z=(x,\xi)$.  The
 representations $\rho$ and $\mu$ can be combined and give rise to
 the  extended metaplectic representation $\mu_e$ of the group
 $G=\H^d\rtimes Sp(d,\bR)$, the semidirect product of $\H^d$ and
 $Sp(d,\R)$. The group law on $G$ is
 \begin{equation}\label{Glaw}
 \left((z,t),A\right)\cdot\left((z',t'),A'\right)
 =\left((z,t)\cdot(Az',t'),AA'\right)
 \end{equation}
 and the extended metaplectic   representation $\mu_e$ of $G$ is
 \begin{equation}\label{defex}
 \mu_e\left((z,t),A\right)=\rho(z,t)\circ\mu(A).
 \end{equation}
 The role of the center of the
 Heisenberg group is only a product by a phase factor, and if we omit it,  the ``true'' group
 under consideration is $\rdd\rtimes Sp(d,\bR)$, which we denote again
 by $G$. Thus $G$ acts naturally by affine transformations on phase
 space, namely
 \begin{equation}
 g\cdot(x,\xi)=\left((q,p),A\right)\cdot (x,\xi)= A {}^t (x,\xi)+{}^t(q,p).
 \label{affaction3}
 \end{equation}
 and its Lie algebra is isomorphic to  the Lie algebra $\mathcal{P}_2$ (of real-valued polynomials of degree $\leq 2$).
Indeed, the map which takes the infinitesimal (extended) metaplectic operator $i a^w$, with ($a\in \mathcal{P}_2$) to the hamiltonian vector field
$$\sum_{j=1}^d\left(\frac{\partial a}{\partial x_j}\frac{\partial }{\partial \xi_j} - \frac{\partial a}{\partial \xi_j}\frac{\partial }{\partial x_j}\right)$$
 is an isomorphism between the algebra of infinitesimal (extended) metaplectic operators and the Lie algebra spanned by the constant and linear symplectic vector fields.

 Another generalization is related to time-dependent symbols $a_t$.  Suppose we are given a continuous curve $i a_t^w$, $a_t\in  \mathcal{P}_2$, of infinitesimal (extended) metaplectic operators. Then, the theory  above, written for the case of a time-independent operator $i a^w$) still works and the Schr\"{o}dinger equation \eqref{C12} (with symbol $a$ replaced by $a_t$) has solution than can be  uniquely represented (up to a phase factor) by the extended metaplectic operator
 $$u=e^{ita_t^w}u_0=\mu_e(g_t)u_0$$
  where $g_t\in G$ is a continuous curve on $G$.
 This more general case was studied by Weinstein \cite[Section 3]{Wein85} and  is also object of our further investigations.

\section{Perturbed Schr\"odinger Equations}
We now consider the Cauchy problem in \eqref{C1}, where the  hamiltonian  $a^w$,  Weyl quantization of a real-valued homogeneous  quadratic polynomial as discussed in the previous section, is perturbed by adding a pseudodifferential operator $\sigma^w$  with a symbol $\sigma\in M^{\infty}_{1\otimes v_s}(\rdd)$, $s>2d$. Results concerning the  quantization $a^w$  of a more general real-valued polynomial of degree $\leq 2$ and a potential $V(t,x)$ which is only a multiplication operator are also mentioned in the end of this section.
Our main result is as follows.
\begin{theorem}\label{teofinal}
Consider the Cauchy problem \eqref{C1} with $H = a^w + \sigma^w$ and $a^w$ and $\sigma^w$ as above. Then,\\
 (i) The evolution operator $e^{it H}$ is a generalized metaplectic operator for every $t\in\bR$.
Specifically, we have
\begin{equation}\label{evolution}
e^{itH} = \mu(\mathcal{A}_t)b^w_{1,t}=b^w_{2,t}\mu(\mathcal{A}_t),\quad t\in\bR
\end{equation}
for some symbols $b_{1,t},b_{2,t}\in M^{\infty}_{1\otimes v_s}(\rdd)$ and where $\mu(\mathcal{A}_t)=e^{i t a^w}$ is the solution to \eqref{C12Mt}.\\
(ii) Consider $m\in\cM_{v_s}$,  $1\leq p\leq \infty$. Then $e^{itH}$ extends to a bounded operator from $M^p_{m\circ\A_t}(\rd)$ to $M^p_{m}(\rd)$. In particular,  if $ u_0\in M^p_{v_s}$, then $u(t,\cdot )
= e^{itH}u_0\in  M^p_{v_s}$, for all $t\in\bR$.\end{theorem}
\begin{proof}
The pattern of the proof follows  \cite[Theorem 4.1]{MetapWiener13}, where symbols in the classes $M^{\infty,1}_{1\otimes v_s}(\rdd)$ were considered. The main ingredient is the theory about bounded perturbation of operator (semi)groups (see, e.g. the textbooks~\cite{RS75} and \cite{EN06}) and is quite standard. Indeed, this is the model already employed in some of the papers that inspired our work, namely  Zelditch  \cite{zelditch} and Weinstein \cite{Wein85}.
In what follows we shall explain the main ideas of this proof. \par

Since $\sigma \in  M^{\infty}_{1\otimes s} (\rdd ) \subset  M^{\infty,1} (\rdd ) $, for $s>2d$,  the Weyl operator  $\sigma ^w$ is bounded on $L^2(\rd)$ (see \cite{wiener31} and Theorem \ref{fund} before) and more generally on every modulation space $M^p_m(\rd )$ with weight $m$ being $v_s$-moderate (see \cite{charly06} and Theorem \ref{fund} before). Using the boundedness result for classical metaplectic operators in \eqref{contmua}
we observe that the solution of the unperturbed problem $\mu(\A_t)=e^{i t a^w}$ is a one-parameter group strongly continuous on $L^2(\rd)$ and  on every $M^p_m(\rd )$
with $m \asymp m\circ \cA $ for every $\cA \in Sp(d,\R )$.
Hence also  the evolution  $e^{i t H}$ generates a one-parameter group strongly continuous on  $M^p_m(\rd )$ as above, thanks to the standard theory (see \cite[Ch.~3, Cor.~1.7]{EN06}). This immediately gives the boundedness $\|e^{itH}u_0\|_{M^p_{v_s}}\leq C \|u_0\|_{M^p_{v_s}}$, for all $s,t\in\bR$, that is
the second part of Theorem \ref{teofinal}, $(ii)$; whereas the continuity for a more general weight $m\in\cM_{v_s}$ follows by Theorem \ref{T31}, once we have proved that $e^{i t H}$ is a generalized metaplectic operator. We focus on this claim.
The two main ingredients that give the result are represented by
the algebra property of generalized metaplectic operators contained in Theorem \ref{prod1}, $(i)$,  and the characterization of pseudodifferential operators in Proposition \ref{charpsdo}, $(ii)$. Namely, if $e^{itH}=\mu(\A_t) P(t)$, we need to show that for every $t\in\bR$ the operator $P(t)$ is a pseudodifferential operator with symbol in $M^{\infty}_{1\otimes v_s}(\rdd)$. Writing $$B(t)=\mu(\A_{-t})\sigma^w\mu(\A_t)\in FIO(\A_{-t}\circ Id\circ \A_t,s )=FIO(Id,s ),$$ (hence $B(t)$ is a pseudodifferential operator with symbol in $M^\infty_{1\otimes v_s}$) the operator $P(t)$
can be written using the following \emph{Dyson-Phillips expansion}
\begin{equation}
  \label{eq:kh8}
  P(t) = \mathrm{Id} + \sum _{n=1}^\infty (-i)^n \int _0^t \int
  _0^{t_1} \dots \int
  _0^{t_{n-1}} B(t_1) B(t_2) \dots B(t_n) \, dt_1 \dots dt_n := \sum
  _{n=0}^\infty P_n(t) \, .
\end{equation}
We study the Gabor matrix of $P(t)$, working on a kernel level, as done in the pioneering work of Zelditch  \cite{zelditch} whereas we notice that Weinstein \cite{Wein85} uses the Weyl calculus of pseudodifferential operators and studies a similar  Dyson expansion on a symbol level.

First, using \eqref{eq:kh3} we have
$$
|\langle \prod _{j=1}^n B(t_j) \pi (z) g, \pi (w) g\rangle | \leq
C C_{t_1}C_{t_2}\dots  C_{t_n} \la w-z\ra^{-s} \,
$$
with $C_{t_j}\asymp \|\sigma \circ \A _{t_j}
\inv \|_{M^\infty_{1\otimes v_s}}$, $j=1,\dots n$.
Using Lemma~\ref{lkh1}
\begin{align*}
  \sup _{0\leq r \leq t}  \|\sigma \circ \A _{r}
  \inv \|_{M^\infty_{1\otimes v_s}} & \leq
  \|\sigma \|_{M^\infty_{1\otimes v_s} } \sup _{0\leq r \leq t} \|\A _r \|^s \,  \| V_{\Phi \circ
  \A _r} \Phi \|_{L^1_{v_s}} \leq M(t) \|\sigma \|_{M^\infty_{1\otimes v_s}},
\end{align*}
where $M(t) = \sup _{0\leq r \leq t} \|\A _r \|^s \,  \| V_{\Phi \circ
  \A _r }\Phi \|_{L^1_{v_s}}$ is easily proved to be finite.
The Gabor matrix of $P_n(t)$ can then by controlled by
$$|\langle P_n(t) \pi (z)g, \pi (w)g\rangle |\leq C\frac{t^n}{n!}M(t)^n\|\sigma \|^n_{M^\infty_{1\otimes v_s}} \la w-z\ra^{-s},
$$
and, consequently, the Gabor matrix of $P(t)$ satisfies
\begin{align*}
|\langle P(t) \pi (z) g , \pi (w)g\rangle | & \leq   \sum
_{n=0}^\infty |\langle P_n(t) \pi (z) g , \pi (w)g\rangle |  \la w-z\ra^{-s}\\
 & \leq   C\sum _{n=0}^\infty
   \frac{t^nM(t)^n \|\sigma \|^n_{M^\infty_{1\otimes v_s}}}{n!} \la w-z\ra^{-s}\\
   &=C(t)\la z-w\ra^{-s},
   \end{align*}
for a new function $C(t)>0$.
This gives by  Proposition~\ref{charpsdo}, $(ii)$, that $P(t) = b_{1,t} ^w$ for a symbol $b_{1,t}$   in $M^\infty_{1\otimes v_s}(\rdd)$. Finally, the characterization of generalized metaplectic operators in Theorem \ref{pseudomu} gives also the second equality in \eqref{evolution}:
$$
e^{itH} =  b_{2,t}^w \mu (\A _t)
$$
for some $b_{2,t}\in M^\infty_{1\otimes v_s}(\rdd)$, and this ends the proof.
\end{proof}

We now compare these issues with other results in the literature. First, in the frame-work of time-frequency analysis, we  mention the pioneering work \cite{Benyi} on boundedness on modulation spaces for Fourier multipliers: as special example we find the free particle evolution operator $e^{i t \Delta}$.
The study of PDEs and in particular of nonlinear Schr\"{o}dinger equations in modulation spaces has been widely developed by B. Wang and collaborators in many papers, see e.g., \cite{RSW,baoxiang,wh} and the recent textbook \cite{Wangbook}.  Inspired by these new nonlinear topics we would like to use the previous techniques for a Cauchy problem \eqref{C1} where the operator $H$ contains a nonlinearity, this is our future project.

Estimates on modulation spaces for Schr\"{o}dinger evolution operators $e^{itH}$ were also performed in the works \cite{kki1,kki2,kki3,kki4}, as detailed in the Introduction. \par
A  result similar to Theorem \ref{teofinal}, were  the linear affine transformation $\A_t$ is replaced by a more general symplectomorphism $\chi_t$ and the solution $e^{it H}$ is no more a generalized metaplectic operator but a Fourier integral operator in the classes $FIO(\chi_t,s)$, is provided in \cite[Theorem 4.1]{wavefrontsetshubin13}). There the Hamiltonian  $a^w(x,D)$ is a pseudodifferential operator
 where the symbol $a(z)$,
$z=(x,\xi)$, is real-valued positively homogeneous of degree 2,
i.e.\ $a(\lambda z)=\lambda^2 a(z)$ for $\lambda>0$, with
$a\in\cC^\infty (\rdd\setminus{0})$. Indeed, the
singularity at the origin of $a(z)$ can be admitted as well, by
absorbing it in a non-smooth potential. In this way, the pseudodifferential
operator $a^w(x,D)$ can be modified such that its symbol is in the Shubin classes \cite{Shubin91} (see also \cite{helffer84}) and the symbolic calculus can be applied. 

Finally, we spend some additional words about two papers that we often mentioned in the preceding pages, namely the work of Zelditch  \cite{zelditch} and its generalization by Weinstein \cite{Wein85}. The former studies propagation of singularities for evolution operators $e^{i t H}$ where $H=\Delta-V$ and $V$ is a multiplication by a potential function $V(x)$ which differs only slightly from a positive-definite quadratic function (so $H$ is very similar to the harmonic oscillator). A local representation of the propagator $e^{i t H}$ using metaplectic operators is obtained  and the main result says that the singularities of the solution $u(t,\cdot)=e^{i t H}u_0$ behave as in  the case of the harmonic oscillator. The latter paper generalizes the former one by considering  $H=a_t^w+V_t$, where the time-dependent hamiltonian $a_t^w$  is the  Weyl quantization of a real-valued polynomial $a(x,\xi)$ of degree $\leq 2$ and the time-dependent potential $V_t(x)$ is a smooth function in $(t,x)\in \bR\times\bR^d$ which belongs to particular subclasses of $S^0_{0,0}$. A local representation of $e^{i t H}$ as product of a metaplectic and a  pseudodifferential operator having symbol in these subclasses is obtained and propagation of singularities are  studied.

\section{Propagation of singularities}
To have a complete understanding of the subject, we present related results concerning the propagation of singularities, obtained in \cite{{wavefrontsetshubin13}} . There,  a new definition of wave front set is given,  extending the so-called global wave front set introduced by H\"ormander in   1991 \cite{hormanderglobalwfs91} and which is different from the classical one, already studied in \cite[Chap. 8]{hormander3}, which suits well in the  study of certain classes of  evolution operators of hyperbolic type.\par
First, we recall the classical H\"ormander wave front set in \cite[Chap. 8]{hormander3}.  Given $x_0\in\rd$, we define by $\f_{x_0}$ a test function in $\cC_0^\infty(\rd)$ such that $0\leq \f_{x_0}(x)\leq 1$ for every $x\in\rd$ and $\f_{x_0}(x)=1$ for $x$ in a neighborhood of $x_0$.  Given $\xi_0\in\rd\setminus\{0\}$ we define by $\psi_{\xi_0}$ a function in $\cC^\infty(\rd)$, supported in a conic open set $\Gamma\subset\rd\setminus\{0\}$ containing $\xi_0$, such that $\psi_{\xi_0}(\xi)=1$ for $\xi \in \Gamma'$, $|\xi|\geq A$ for a conic open set $\Gamma'$ such that $\xi_0\in\Gamma'\subset\Gamma$ and $\psi_{\xi_0}(\xi)=0$ for $|\xi |<R$ for some $0<R<A$. We define  the classical H\"ormander wave front set $W F_\psi (u)$ of a distribution $u\in\cS'(\rd)$ (or $u\in\mathcal{D}'(\rd)$), as follows: for $(x_0,\xi_0)\in\rd\times\rd\setminus\{0\}$,  $(x_0,\xi_0)\notin W F_\psi (u)$ if there exist $\f_{x_0}$ and $\psi_{\xi_0}$ such that
$\psi_{\xi_0}(D)(\f_{x_0} u)\in \cS(\rd)$. Here $\psi_{\xi_0}(D)$ is the Fourier multiplier with symbol $\psi_{\xi_0}$, that is
$$ \psi_{\xi_0}(D)(\f_{x_0} u)(x)= \intrd e^{2\pi i x \xi} \psi_{\xi_0}(\xi) \widehat{\f_{x_0} u}(\xi)\,d\xi.
$$
Hence the functions $\f_{x_0}$ and $\psi_{\xi_0}$  represent the cut-off in time and in frequency, respectively. For hyperbolic equations of the type
$$i\frac{\partial u} {\partial t} + a^w(t, x, D)u = 0$$
with $u(0,x)=u_0$, and where  $a(t, x, \xi)$ is a real-valued hamiltonian, homogeneous of the first order in $\xi$, we have for the solution $u(t,x)$:
$$WF_\psi(u(t)) =\chi_t (WF_\psi(u_0))$$
where $\chi_t$ is the symplectomorphism defined by the Hamiltonian $a(t, x, \xi)$.  We want a similar result for the evolution of the Schr\"{o}dinger equation \eqref{C1}. To reach this goal, we introduce a wave front set that does the job. This is a generalization of the global H\"ormander  wave front set $WF_G(f)$, that we recall in what follows. Namely,
we use an equivalent definition  via STFT introduced (and proved to be equivalent) in \cite{RWwavefrontset}.
 Consider $u\in\cS'(\rd)$, $z_0\in\rdd\setminus \{0\}$ and fix $g\in\cS(\rd)\setminus\{0\}$. Then  $z_0\notin WF_G (u)$ if there exists an
open conic set $\Gamma_{z_0}\subset \rdd$ containing $z_0$ such
that for every $r>0$
\begin{equation}\label{WFSeq}
|V_g u(z)|\leq C_r \la z\ra^{-r},\quad z\in \Gamma_{z_0}
\end{equation}
for a suitable $C_r>0$.
Then $WF_G (u)$ is a conic closed subset of $\rdd\setminus \{0\}$ and its definition does not depend on the choice of the nonzero window function $g$ in $\cS(\rd)$.
Finally, we define the  Gabor wave front set $\wpr(u)$ under our consideration as follows \cite{wavefrontsetshubin13}.
\begin{definition}\label{defWFgr}
Let  $g\in\cS(\rd)$, $g\not=0$, $r>0$. For $u\in M^p_{v_{-r}}(\rd)$, $z_0\in \rdd$, $z_0\not=0$, we say that $z_0\notin \wpr (u)$ if there exists an open conic neighborhood $\Gamma_{z_0}\subset \rdd$ containing $z_0$ such that for a suitable constant $C>0$
\begin{equation}\label{5.1}
\int_{\Gamma_{z_0}}|V_gu(z)|^p \la z\ra^{pr}\,dz<\infty
\end{equation}
(with obvious changes for $p=\infty$).
\end{definition}
Then $\wpr (u)$ is well-defined as conic closed subset of $\rdd\setminus\{0\}$. Furthermore, the definitions of  $\wpr (f)$  does not depend on the choice of the window $g$. We have the following characterizations:
 $$u\in M^p_{v_r}(\rd) \Leftrightarrow \wpr (u)=\emptyset$$
 and, similarly,
 $$u\in\cS(\rd) \Leftrightarrow WF_G (f)=\emptyset.$$
The propagation of singularities for the evolution $e^{i t H}$ of our equation \eqref{C1}  are proved in \cite{wavefrontsetshubin13} and summarized as follows.
\begin{theorem}\label{T1.5}
Consider $\sigma\in M^{\infty}_{1\otimes v_s}$, $s>2d$, $1\leq p\leq\infty$. Then
\begin{equation}\label{T1.5eq1}
e^{i t H}: M^p_{v_r}(\rd)\to M^p_{v_r} (\rd)
\end{equation}
continuously, for $|r|< s-2d$. Moreover,  for $u_0\in M^p_{v_{-r}}(\rd)$,
\begin{equation}\label{T1.5eq2}
WF^{p,r}_G(e^{i t H} u_0)=\A_t(WF^{p,r}_G (u_0)),
\end{equation}
provided $0<2r<s-2d$.
\end{theorem}
Observe the more restrictive assumption  on $r$ for \eqref{T1.5eq2}, with respect to that for
\eqref{T1.5eq1}.

Using \eqref{inteSs} and \eqref{T1.5eq2}, we may recapture the known results for the propagation in the case of a smooth potentials, i.e. $\sigma\in S^0_{0,0}$.
Then the estimate \eqref{asterisco} is satisfied for every $s$ and from Theorem \ref{T1.5} we recapture for $u_0\in\cS'(\rd)$
\begin{equation}\label{WT1.3}
WF_G(e^{i t H} u_0)=\A_t(WF_G (u_0)).
\end{equation}
This identity is contained in many preceding results. Although it is impossible to do justice to the vast literature in this
connection, let us mention the pioneering work of H\"{o}rmander \cite{hormanderglobalwfs91} 1991, who defined the  wave front set in \eqref{WFSeq} as well as its analytic version,  and proved  \eqref{WT1.3} in the case of the metaplectic operators.
For subsequent results providing \eqref{WT1.3} and its analytic-Gevrey version for general smooth symbols, let us refer to \cite{Hassel-Wunsch,ito,ito-nakamura,Martinez,Martinez2,Mizuhara,Nakamura,Nakamura2,wunsch}. The wave front sets introduced there under different names actually coincide with those of  H\"{o}rmander 1991, cf. \cite{RWwavefrontset}, \cite{sw} and  \cite{Cappiello-shulz}.
Finally, for sake of completeness, let us recall the propagation of singularities for a pseudodifferential operators in the framework of the global and the Gabor wave front set. Observe the difference in the symbol classes and in the domain of the distribution $u$.
\begin{proposition}\label{5.4}
Let $\sigma \in  M^\infty_{1\otimes v_s}(\rdd)$, $s>2d$ and $0<2r<s-2d$. Then, for every $u\in M^p_{v_{-r}}(\rd)$ we have
\begin{equation}\label{5.10eq}
\wpr \,(\sigma(x,D)u)\subset \wpr (u).
\end{equation}
If $\sigma\in S^0_{0,0}$, then for every $u\in\cS'(\rd)$,
\begin{equation}\label{5.11eq}
WF_G  \,(\sigma(x,D)u)\subset WF_G(u).
\end{equation}
 \end{proposition}
From the previous results it follows that the study of the propagation of singularities of the evolution $e^{i t H}$ should be conducted as follows: if  the perturbation $\sigma^w$ in \eqref{C1bis} is the quantization of a smooth potential $\sigma\in S^0_{0,0}$ we use the global wave front set $WF_G(u)$, otherwise, if the symbol is rough and $\sigma\in  M^\infty_{1\otimes v_s}(\rdd)$, for some $s>2d$, then we use the Gabor wave front set  $\wpr (u)$, with the limitation $0< 2 r< s-2d$.
 We end up this section with two examples, the first  is the harmonic oscillator and the second example is a perturbation of it with a rough potential. In order to compute the wave front set $WF_G(u)$ we need the following preliminary results.
  \begin{proposition}\label{5.5}
 We have:\\
 (i) Let $\xi_0$ be fixed in $\rd$. Then $$WF_G\, (e^{2\pi i  x\xi_0 })=\{z=(x,\xi), x\not=0, \xi=0\}$$ independently of $\xi_0$.\\
 (ii) Let $c\in\bR$, $c\not=0$, be fixed. Then
  $$WF_G\, (e^{\pi i c|x|^2 })=\{z=(x,\xi), \, x\not=0, \, \xi=c x\}.$$
  \end{proposition}
  \textit{Example 1.} {\bf The harmonic oscillator}.\par
  Consider the Cauchy problem
  \begin{equation}\label{5.17eq}
  \begin{cases} i \partial_t
   u +\frac{1}{4\pi}\Delta u-\pi|x|^2 u=0\\
  u(0,x)=u_0(x).
  \end{cases}
  \end{equation}
  The solution is
  \begin{equation}\label{5.18eq}
  u(t,x)=(\cos t)^{-d/2}\intrd e^{2\pi i [\frac1{\cos t} x\xi - \frac{\tan t}2 (x^2+\xi^2)]}\hat{u_0}(\xi)\,d\xi,\quad t\not=\frac\pi 2 + k\pi,\,\,k\in\bZ.
  \end{equation}
  The Gabor matrix with Gaussian window $g(x)=e^{-\pi|x|^2}$ can be explicitly computed as
  \begin{equation}\label{5.19eq}
  |k(w,z)|=2^{-\frac d 2} e^{-\frac \pi 2 |z-\A_t(w)|^2},
  \end{equation}
  where
  \begin{equation}\label{5.20eq}
\A_t{}^t(y,\xi)=\begin{pmatrix}(\cos
t)I&(-\sin t)I\\(\sin t)I&(\cos
t)I\end{pmatrix} \begin{pmatrix}y\\\xi\end{pmatrix}
\end{equation}
with $I$ being the identity matrix.
  Observe that the expression \eqref{5.19eq} is meaningful for every $t\in\bR$. Let us address to \cite[Section 6.2]{fio3} for applications to numerical experiments. \par
  We may test \eqref{WT1.3} on the initial datum $u_0(x)=1$, giving for $t<\pi/2$,
  $$u(t,x)=(\cos t)^{-d/2} e^{-\pi i \tan t |x|^2}.
  $$
  From Proposition \ref{5.5}, $(i)$ and $(ii)$, we have coherently with \eqref{5.20eq}
  \begin{align*}WF_G \,(u(t,x))&=\{ (x,\xi),\,x=(\cos t)y, \,\xi=(\sin t) y, y\not=0\}\\
  &=\A_t(WF_G\,(1))=\A_t (\{(y,\eta),\, y\not=0,\eta=0\}).
  \end{align*}
  \textit{Example 2.}\label{ex4} {\bf Perturbed harmonic oscillator}.\par
Consider the perturbed harmonic oscillator in dimension $d=1$
\begin{equation}\label{5.31eq}
\begin{cases} i \frac{\partial u }{\partial t}
  +\frac{1}{4\pi}\frac{\partial^2 u}{\partial x^2}-\pi x^2 u+ |\sin x|^\mu u=0\\
u(0,x)=u_0(x)
\end{cases}
\end{equation}
with $\mu>1$, $x\in\R$.  So in this case $\sigma^w=|\sin x|^\mu$ is a multiplication operator.  By \cite[Corollary 2.4]{wavefrontsetshubin13} we have $\sigma\in M^\infty_{1\otimes v_{\mu+1}}(\R^2)$.
 From Theorem \ref{T1.5} we have that the Cauchy problem is well-posed for $u_0\in M^p_{v_{r}}(\R)$, $|r|<\mu-2$ and the propagation of $\wpr\, (u(t,\cdot))$ for $t\in \bR$ takes place as in Example $1$, for $0<r<\mu/2-1$.\par
 Similar examples are easily obtained for the free particle \eqref{cp}, by using \eqref{zz1} and \eqref{4bis}, which we proposed as motivation of our time-frequency approach.
  \vskip0.3truecm

\section*{Acknowledgment}
We would like to thank Professor K.~Gr{\"o}chenig  for inspiring this work.

\end{document}